\pdfoutput=1
\documentclass[12pt,letter]{amsart}
\usepackage{xcolor}
\usepackage[all,color]{xy}
\usepackage{amsmath,amssymb,amsthm,amscd,amsxtra,amsfonts,mathrsfs,graphicx,enumerate,bm,slashed,array,mathtools,setspace,amsfonts,mathtools}
\input xy
\xyoption{all}
\newtheorem{Theorem}{Theorem}[section]
\newtheorem{Lemma}[Theorem]{Lemma} 
\newtheorem{Proposition}[Theorem]{Proposition}

\newtheorem{Example}[Theorem]{Example}

\newtheorem{Definition}[Theorem]{Definition}

\newtheorem*{Theorem A}{Theorem A}

\newcommand*{\overbar}[1]{\mkern 1.5mu\overline{\mkern-1.5mu#1\mkern-1.5mu}\mkern 1.5mu}
\advance\evensidemargin-.5in
\advance\oddsidemargin-.5in
\advance\textwidth1in

\begin{document}
\author{Charlie Beil}
 \address{Institut f\"ur Mathematik und Wissenschaftliches Rechnen, Universit\"at Graz, Heinrichstrasse 36, 8010 Graz, Austria.}
 \email{charles.beil@uni-graz.at}
 \title[Nonnoetherian coordinate rings with unique maximal depictions]{Nonnoetherian coordinate rings with\\ unique maximal depictions}
 \date{}
 
\begin{abstract}
A depiction of a nonnoetherian integral domain $R$ is a special coordinate ring that provides a framework for describing the geometry of $R$. 
We show that if $R$ is noetherian in codimension 1, then $R$ has a unique maximal depiction $T$.
In this case, the geometric dimensions of the points of $\operatorname{Spec}R$ may be computed directly from $T$.
If in addition $R$ has a normal depiction $S$, then $S$ is the unique maximal depiction of $R$. 
\end{abstract}

\maketitle

\section{Introduction}

In this article all algebras are assumed to be commutative integral domains over an algebraically closed base field $k$. 
Depictions were introduced in \cite{B3} to provide a framework for describing the geometry of nonnoetherian algebras with finite Krull dimension. 
A depiction of a nonnoetherian algebra $R$ is a finitely generated algebra $S$ that is as close as possible to $R$, in a suitable geometric sense (Definition \ref{depiction def}). 
In this framework, the geometry of the maximal spectrum $\operatorname{Max}R$ is viewed as the algebraic variety $\operatorname{Max}S$, together with a collection of algebraic sets of $\operatorname{Max}S$ which are identified as `smeared-out' positive dimensional closed points \cite{B2}. 

Depictions have played an essential role in understanding the algebraic and representation theoretic properties of a class of quiver algebras called dimer algebras \cite{B1,B4,B5}.
However, there are many open questions regarding the fundamental nature of depictions; for example, it is not known whether every subalgebra of a finite type integral domain admits a depiction, or whether every depiction is contained in a maximal depiction.
Here we consider the question: \textit{What algebras admit unique maximal depictions?}

In general, maximal depictions need not be unique.
Indeed, consider the rings
$$S = k[x,y,z] \ \ \ \text{ and } \ \ \ R = k + xyS.$$
Then the overrings 
$$S[x^{-1}] \ \ \text{ and } \ \ S[y^{-1}]$$
are both depictions of $R$, whereas their minimal proper overring $S[x^{-1}, y^{-1}]$ is not \cite[Proposition 3.19]{B3}.
To identify a class of algebras that admit unique maximal depictions, we introduce the following definition.

\begin{Definition} \label{noetherian in codimension 1} \rm{
We say $R$ is \textit{noetherian in codimension 1} if $R$ admits a depiction $S$ such that each codimension 1 subvariety of $\operatorname{Max}S$ intersects the open set 
$$U_{S/R} := \left\{ \mathfrak{n} \in \operatorname{Max}S \ | \ R_{\mathfrak{n}\cap R} = S_{\mathfrak{n}} \right\}.$$
}\end{Definition}

We will show that this definition is independent of the choice of depiction $S$ (Proposition \ref{independent}). 
Furthermore, if $R$ is noetherian in codimension $1$, then for each height $1$ prime $\mathfrak{q} \in \operatorname{Spec}S$, the localization $R_{\mathfrak{q}\cap R}$ is noetherian (Lemma \ref{lemma 0}.3).

Our main theorem is the following.

\begin{Theorem} \label{main} (Theorems \ref{geometric height} and \ref{normal theorem}.)
Suppose $R$ is noetherian in codimension 1.
Let $S$ be any depiction of $R$, and consider the global sections ring,
$$T := \bigcap_{\mathfrak{n} \in U_{S/R}} S_{\mathfrak{n}} = \Gamma(U_{S/R}).$$
\begin{enumerate}
 \item $T$ is the unique maximal depiction of $R$.
In particular, $T$ is independent of the choice of depiction $S$.
 \item For each $\mathfrak{p} \in \operatorname{Spec}R$ there is some $\mathfrak{t} \in \operatorname{Spec}T$ lying over $\mathfrak{p}$ such that the geometric dimension of $\mathfrak{p}$ equals the Krull dimension of $T/\mathfrak{t}$,
$$\operatorname{gdim} \mathfrak{p} = \operatorname{dim}T/\mathfrak{t}.$$
 \item Let $\overbar{S}$ be the normalization of $S$.
Then
$$S \subseteq T \subseteq \overbar{S}.$$
In particular, if $S = \overbar{S}$ is normal, then $S$ is the unique maximal depiction of $R$, as well as the unique normal depiction of $R$.
\end{enumerate}
\end{Theorem}

For example, consider the family of algebras
$$S_j := k[x,y,xz,yz,xz^2, yz^2, \ldots, xz^{j-1},yz^{j-1}, z^j] \ \ \ \text{ and } \ \ \ R := k + (x,y)S_1,$$
where $j \geq 1$, and $(x,y)S_1$ is the ideal of $S_1 = k[x,y,z]$ generated by $x$ and $y$.
Each $S_j$ is a depiction of $R$, and $R$ is noetherian in codimension $1$ (Example \ref{j geq 1}).
Since $S_1$ is normal, Theorem \ref{main} implies that $S_1$ is the unique maximal depiction of $R$.

Claim (2) in Theorem \ref{main} provides a means of computing the geometric dimension of a point of $\operatorname{Spec}R$ in the case $R$ is noetherian in codimension $1$. 
If a depiction has the property given in Claim (2), then we say it is \textit{saturated}.

In Section \ref{specialsection}, we consider the special case where $R$ has the form $R = k + I$, with $I$ a nonzero radical ideal of $S$ (and $R$ is not necessarily noetherian in codimension $1$).
Using Theorem \ref{main}, we show that if $\operatorname{dim}S/I \geq 1$, then $S$ is a saturated depiction of $R$ (Theorem \ref{wind}).

We conclude with a few examples of maximal depictions in Section \ref{examplesection}.
Notably, we show that if $R$ admits a unique maximal depiction $S$ but is not noetherian in codimension 1, then in general the geometric dimension of a point $\mathfrak{p} \in \operatorname{Spec}R$ need not equal the Krull dimension of $S/\mathfrak{q}$, for any $\mathfrak{q} \in \operatorname{Spec}S$ over $\mathfrak{p}$ (Example \ref{not 1}).

\section{Preliminary definitions} \label{terminology}

Let $S$ be an integral domain and a finitely generated $k$-algebra, and let $R$ be a (possibly nonnoetherian) subalgebra of $S$.
Denote by $\operatorname{Max}S$, $\operatorname{Spec}S$, and $\operatorname{dim}S$ the maximal spectrum (or variety), prime spectrum (or affine scheme), and Krull dimension of $S$ respectively; similarly for $R$. 
%
%
For a subset $I \subset S$, set $\mathcal{Z}_S(I) := \left\{ \mathfrak{n} \in \operatorname{Max}S \ | \ \mathfrak{n} \supseteq I \right\}$.

We will consider the following subsets of $\operatorname{Max}S$ and $\operatorname{Spec}S$, 
\begin{equation*} \label{noetherian locus}
\begin{aligned}
U_{S/R} := & \left\{ \mathfrak{n} \in \operatorname{Max}S \ | \ R_{\mathfrak{n}\cap R} = S_{\mathfrak{n}} \right\},\\
\tilde{U}_{S/R} := & \left\{ \mathfrak{q} \in \operatorname{Spec}S \ | \ R_{\mathfrak{q} \cap R} = S_{\mathfrak{q}} \right\},\\
Z_{S/R} := & \left\{ \mathfrak{q} \in \operatorname{Spec}S \ | \ \mathcal{Z}_S(\mathfrak{q}) \cap U_{S/R} \not = \emptyset \right\}.
\end{aligned}
\end{equation*}
Note that if $U_{S/R} \not = \emptyset$, then $R$ and $S$ have the same fraction field: if $\mathfrak{n} \in U_{S/R}$, then
\begin{equation} \label{goodbye}
\operatorname{Frac}R = \operatorname{Frac}(R_{\mathfrak{n} \cap R}) = \operatorname{Frac}(S_{\mathfrak{n}}) = \operatorname{Frac}S.
\end{equation}

\begin{Definition} \label{depiction def} \rm{\cite[Definition 3.1]{B3}
\begin{itemize}
 \item We say $S$ is a \textit{depiction} of $R$ if the morphism
$$\iota_{S/R}: \operatorname{Spec}S \rightarrow \operatorname{Spec}R, \ \ \ \ \mathfrak{q} \mapsto \mathfrak{q} \cap R,$$ 
is surjective, and
\begin{equation} \label{condition}
U_{S/R} = \left\{ \mathfrak{n} \in \operatorname{Max}S \ | \ R_{\mathfrak{n} \cap R} \text{ is noetherian} \right\} \not = \emptyset.
\end{equation}
 \item The \textit{geometric height} of $\mathfrak{p} \in \operatorname{Spec}R$ is the minimum
$$\operatorname{ght}(\mathfrak{p}) := \operatorname{min} \left\{ \operatorname{ht}_S(\mathfrak{q}) \ | \ \mathfrak{q} \in \iota^{-1}_{S/R}(\mathfrak{p}), \ S \text{ a depiction of } R \right\}.$$
The \textit{geometric dimension} of $\mathfrak{p}$ is
$$\operatorname{gdim} \mathfrak{p} := \operatorname{dim}R - \operatorname{ght}(\mathfrak{p}).$$
 \item A depiction $S$ of $R$ is \textit{maximal} if $S$ is not properly contained in another depiction of $R$. 
 \end{itemize}
} \end{Definition}

If $R$ is fixed, then we will often write $\iota_S$ for $\iota_{S/R}$.

\section{Proof of main theorem}

Throughout, let $S$ and $S'$ be depictions of $R$.
We begin by recalling the following useful facts from \cite{B3}.

\begin{Lemma} \label{lemma 0}
We have
\begin{enumerate}
 \item $\operatorname{dim}R = \operatorname{dim}S$.
 \item The locus $U_{S/R}$ is an open dense subset of $\operatorname{Max}S$.
 \item $Z_{S/R} \subseteq \tilde{U}_{S/R}$.
 \item If $\mathfrak{q} \in \tilde{U}_{S/R}$, then 
 $$\iota_{S}^{-1}\iota_{S}(\mathfrak{q}) = \left\{ \mathfrak{q} \right\}.$$
 \item The images of the loci $U_{S/R}$ and $U_{S'/R}$ in $\operatorname{Max}R$ coincide,
$$\iota_{S}(U_{S/R}) = \iota_{S'}(U_{S'/R}).$$
\end{enumerate}
\end{Lemma}

\begin{proof}
The claims are respectively \cite[Theorem 2.5.4; Proposition 2.4.2; Lemma 2.2; Theorem 2.5.1; Theorem 3.5]{B3}.
\end{proof}

\begin{Lemma} \label{lemma 1}
If $\mathfrak{q} \in Z_{S/R}$, then there is a unique prime $\mathfrak{q}' \in \operatorname{Spec}S'$ such that
\begin{equation*}
\mathfrak{q}' \cap R = \mathfrak{q} \cap R.
\end{equation*}
Moreover,
\begin{equation*}
\mathfrak{q}' \in Z_{S'/R} \ \ \ \ \text{ and } \ \ \ \ \operatorname{ht}_{S'}(\mathfrak{q}') = \operatorname{ht}_{S}(\mathfrak{q}).
\end{equation*}
\end{Lemma}

\begin{proof}
Suppose the hypotheses hold.
Since $\mathcal{Z}_S(\mathfrak{q}) \cap U_{S/R} \not = \emptyset$, there is some $\mathfrak{n} \in U_{S/R}$ for which $\mathfrak{n} \supseteq \mathfrak{q}$. 
Whence
$$\iota_{S}(\mathfrak{n}) \in \iota_{S}(U_{S/R}) \stackrel{\textsc{(i)}}{=} \iota_{S'}(U_{S'/R}),$$
where (\textsc{i}) holds by Lemma \ref{lemma 0}.5.
Thus there is some $\mathfrak{n}' \in U_{S'/R}$ for which
$$\mathfrak{n}' \in \iota_{S'}^{-1}\iota_{S} (\mathfrak{n}).$$
In particular, $\mathfrak{n}' \cap R = \mathfrak{n} \cap R$ and
$$S'_{\mathfrak{n}'} = R_{\mathfrak{n}'\cap R} = R_{\mathfrak{n} \cap R} = S_{\mathfrak{n}}.$$

Now $\mathfrak{q}S_{\mathfrak{n}}$ is a prime ideal of $S_{\mathfrak{n}}$ since $\mathfrak{q}$ is a prime ideal of $S$.
Thus
$$\mathfrak{q}S'_{\mathfrak{n}'} = \mathfrak{q}S_{\mathfrak{n}}$$
is a prime ideal of $S'_{\mathfrak{n}'} = S_{\mathfrak{n}}$.
Whence the intersection
$$\mathfrak{q}' := \mathfrak{q} S'_{\mathfrak{n}'} \cap S'$$
is a prime ideal of $S'$ contained in $\mathfrak{n}'$.
Therefore 
\begin{equation} \label{Zq'}
\mathcal{Z}_{S'}(\mathfrak{q}') \cap U_{S'/R} \not = \emptyset,
\end{equation}
that is, $\mathfrak{q}' \in Z_{S'/R}$. 
Furthermore,
$$\mathfrak{q}' \cap R = \mathfrak{q}S'_{\mathfrak{n}'} \cap S' \cap R = \mathfrak{q}S'_{\mathfrak{n}'} \cap S \cap R = \mathfrak{q}S_{\mathfrak{n}} \cap S \cap R = \mathfrak{q} \cap R.$$
Therefore $\mathfrak{q}' \cap R = \mathfrak{q} \cap R$.
Uniqueness of $\mathfrak{q}' \in \operatorname{Spec}S'$ follows from (\ref{Zq'}) and Lemmas \ref{lemma 0}.3 and \ref{lemma 0}.4.

Finally, the heights of $\mathfrak{q}$ and $\mathfrak{q}'$ coincide:
$$\operatorname{ht}_{S'}(\mathfrak{q}') = \operatorname{dim}S'_{\mathfrak{q}'} \stackrel{\textsc{(i)}}{=} \operatorname{dim}R_{\mathfrak{q}' \cap R} = \operatorname{dim}R_{\mathfrak{q}\cap R} \stackrel{\textsc{(ii)}}{=} \operatorname{dim}S_{\mathfrak{q}} = \operatorname{ht}_{S}(\mathfrak{q}),$$
where (\textsc{i}) and (\textsc{ii}) hold by Lemma \ref{lemma 0}.3.
\end{proof}

\begin{Proposition} \label{lemma 1.1}
If $\mathfrak{q} \in Z_{S/R}$, then 
$$\operatorname{ght}_R(\mathfrak{q} \cap R) = \operatorname{ht}_S(\mathfrak{q}).$$
\end{Proposition}

\begin{proof}
Follows from Lemma \ref{lemma 1}.
\end{proof}

Denote by $T_{S/R}$ the global sections ring on $U_{S/R}$,
$$T_{S/R} := \Gamma(U_{S/R}) = \bigcap_{\mathfrak{n} \in U_{S/R}} S_{\mathfrak{n}}.$$

\begin{Proposition} \label{contains}
The global sections ring $T_{S/R}$ satisfies
$$T_{S/R} := \bigcap_{\mathfrak{n} \in U_{S/R}} S_{\mathfrak{n}} = \bigcap_{\mathfrak{q} \in Z_{S/R}} S_{\mathfrak{q}},$$ 
and contains each depiction of $R$.
\end{Proposition}

\begin{proof}
Given depictions $S$, $S'$ of $R$, we have
$$S' \subseteq \bigcap _{\mathfrak{q}' \in Z_{S'/R}} S'_{\mathfrak{q}'} 
\stackrel{\textsc{(i)}}{=} \bigcap_{\mathfrak{q}' \in Z_{S'/R}} R_{\mathfrak{q}' \cap R}
\stackrel{\textsc{(ii)}}{=} \bigcap_{\mathfrak{q} \in Z_{S/R}} R_{\mathfrak{q} \cap R}
\stackrel{\textsc{(iii)}}{=} \bigcap_{\mathfrak{q} \in Z_{S/R}} S_{\mathfrak{q}} \stackrel{\textsc{(iv)}}{=} \bigcap_{\mathfrak{n} \in U_{S/R}} S_{\mathfrak{n}} = T_{S/R}.$$
Indeed, (\textsc{i}) and (\textsc{iii}) hold by Lemma \ref{lemma 0}.3, and (\textsc{ii}) holds by Lemma \ref{lemma 1}.
(\textsc{iv}) holds since if $\mathfrak{q} \in Z_{S/R}$, then there is some $\mathfrak{n} \in U_{S/R}$ such that $\mathfrak{n} \supseteq \mathfrak{q}$; in particular, $S_{\mathfrak{n}} \subseteq S_{\mathfrak{q}}$.
\end{proof}

Denote by $D_S$ the set of height $1$ prime ideals of $S$,
$$D_S := \left\{ \mathfrak{q} \in \operatorname{Spec}S \ | \ \operatorname{ht}(\mathfrak{q}) = 1 \right\}.$$
Note that, by definition, $R$ is noetherian in codimension $1$ if $R$ admits a depiction $S$ for which $D_S \subseteq Z_{S/R}$.

\textit{For the remainder of this section, we will assume that $R$ is noetherian in codimension 1 unless stated otherwise.} 

\begin{Lemma} \label{contained in normalization}
Suppose $D_S \subseteq Z_{S/R}$.
Then $T_{S/R}$ is contained in the normalization $\overbar{S}$ of $S$.
In particular, $T_{S/R}$ is an integral extension of $S$.
\end{Lemma}

\begin{proof} 
Since $S$ is a noetherian domain, its normalization $\overbar{S}$ is given by \cite[Theorem 11.5.ii]{M} 
\begin{equation} \label{overbar S}
\overbar{S} = \bigcap_{\mathfrak{q} \in D_S} S_{\mathfrak{q}}.
\end{equation}
But $D_S \subseteq Z_{S/R}$ by assumption.
Therefore $T_{S/R} \subseteq \overbar{S}$, by Proposition \ref{contains}. 
\end{proof}

\begin{Proposition} \label{fg}
Suppose $D_S \subseteq Z_{S/R}$.
Then $T_{S/R}$ is finitely generated as an $S$-module and as a $k$-algebra.
\end{Proposition}

\begin{proof}
Set $T := T_{S/R}$.
$S$ is a finitely generated $k$-algebra since $S$ is a depiction of $R$.
Thus its normalization $\overbar{S}$ is a finitely generated $S$-module by the Noether normalization lemma \cite[Corollary 13.13]{E}.
Furthermore, $T$ is a submodule of $\overbar{S}$ by Lemma \ref{contained in normalization}.
Thus $T$ is a finitely generated $S$-module since $S$ is noetherian. 
Therefore $T$ is a finitely generated $k$-algebra.
\end{proof}

\begin{Lemma} \label{surjective morphism1} 
Suppose $D_S \subseteq Z_{S/R}$, and set $T := T_{S/R}$.
The morphism 
$$\iota_{T/S}: \operatorname{Spec}T \to \operatorname{Spec}S, \ \ \ \ \mathfrak{t} \mapsto \mathfrak{t} \cap S,$$
is surjective.
\end{Lemma}

\begin{proof}
$T$ is an integral extension of $S$ by Lemma \ref{contained in normalization}, and therefore $\iota_{T/S}$ is surjective \cite[Theorem 9.3.i]{M}. 
\end{proof}

\begin{Theorem} \label{depiction}
Suppose $D_S \subseteq Z_{S/R}$.
Then $T_{S/R}$ is a depiction of $R$.
\end{Theorem}

\begin{proof}
Set $T := T_{S/R}$.

(i) $T$ is a finitely generated $k$-algebra by Proposition \ref{fg}.

(ii) We claim that  
\begin{equation} \label{U*}
U_{T/R} = \left\{ \mathfrak{t} \in \operatorname{Max}T \ | \ R_{\mathfrak{t} \cap R} \text{ is noetherian} \right\}.
\end{equation}
Consider $\mathfrak{t} \in \operatorname{Max}T$ for which $R_{\mathfrak{t} \cap R}$ is noetherian.
Since $\mathfrak{t} \in \operatorname{Max}T$ and $T$ is a finitely generated $k$-algebra containing $S$, the intersection $\mathfrak{n} := \mathfrak{t} \cap S$ is a maximal ideal of $S$.
Furthermore, 
\begin{equation} \label{time}
\mathfrak{t} \cap R = \mathfrak{t} \cap S \cap R = \mathfrak{n} \cap R.
\end{equation}
Whence $R_{\mathfrak{t} \cap R} = R_{\mathfrak{n} \cap R}$.
Thus $R_{\mathfrak{n} \cap R}$ is noetherian since $R_{\mathfrak{t} \cap R}$ is noetherian.
Therefore $\mathfrak{n} \in U_{S/R}$ since $S'$ is a depiction of $R$.
Consequently,
\begin{equation} \label{R n cap R}
R_{\mathfrak{t} \cap R} = R_{\mathfrak{n} \cap R} = S_{\mathfrak{n}}.
\end{equation}

Furthermore,
\begin{multline} \label{time2}
S_{\mathfrak{n}} = S_{\mathfrak{t} \cap S} \subseteq T_{\mathfrak{t}} = ( \bigcap_{\mathfrak{n}' \in U_{S/R}} S_{\mathfrak{n}'} )_{\mathfrak{t}} \subseteq 
\bigcap_{\mathfrak{n}' \in U_{S/R}} \left( S_{\mathfrak{n}'} \right)_{\mathfrak{t} \cap S_{\mathfrak{n}'}} \\
\stackrel{\textsc{(i)}}{\subseteq} \left( S_{\mathfrak{n}} \right)_{\mathfrak{t} \cap S_{\mathfrak{n}}}
= (S_{\mathfrak{t} \cap S})_{\mathfrak{t} \cap (S_{\mathfrak{t}\cap S})} = S_{\mathfrak{t}\cap S},
\end{multline}
where (\textsc{i}) holds since $\mathfrak{n} \in U_{S/R}$.
Whence $S_{\mathfrak{n}} = T_{\mathfrak{t}}$.
Thus together with (\ref{R n cap R}) we obtain
$$R_{\mathfrak{t} \cap R} = S_{\mathfrak{n}} = T_{\mathfrak{t}}.$$
Therefore $\mathfrak{t} \in U_{T/R}$.
The converse inclusion ($\subseteq$) in (\ref{U*}) is clear.

(iii) The morphism $\iota_{T/R}: \operatorname{Spec}T \to \operatorname{Spec}R$ is surjective since it factors into surjective maps
$$\operatorname{Spec}T \stackrel{\iota_{T/S}}{\longrightarrow} \operatorname{Spec}S \stackrel{\iota_{S/R}}{\longrightarrow} \operatorname{Spec}R.$$ 
Indeed, $\iota_{T/S}$ is surjective by Lemma \ref{surjective morphism1}, and $\iota_{S/R}$ is surjective since $S$ is a depiction of $R$.

(iv) Finally, we claim that $U_{T/R}$ is nonempty.
Since $S$ is a depiction of $R$, there is some $\mathfrak{n} \in U_{S/R}$.
By Lemma \ref{surjective morphism1}, there is some $\mathfrak{t} \in \operatorname{Max}T$ such that $\mathfrak{t} \cap S = \mathfrak{n}$.
Thus
$$R_{\mathfrak{t}\cap R} \stackrel{\textsc{(i)}}{=} R_{\mathfrak{n} \cap R} = S_{\mathfrak{n}} \stackrel{\textsc{(ii)}}{=} T_{\mathfrak{t}},$$
where (\textsc{i}) holds by (\ref{time}), and (\textsc{ii}) holds by (\ref{time2}).
Therefore $\mathfrak{t} \in U_{T/R}$.
\end{proof}

\begin{Lemma} \label{t height 1}
Set $T := T_{S/R}$.
If $D_S \subseteq Z_{S/R}$, then $D_T \subseteq Z_{T/R}$.
\end{Lemma}

\begin{proof}
Let $\mathfrak{t} \in D_T$, and set $\mathfrak{q} := \mathfrak{t} \cap S$.
By Lemma \ref{contained in normalization}, $T$ is an integral extension of $S$.
Thus $\operatorname{ht}_T(\mathfrak{t}) = 1$ implies $\operatorname{ht}_{S}(\mathfrak{q}) = 1$ \cite[Theorem 46]{K}.
Whence $\mathfrak{q} \in Z_{S/R}$ since $S$ is noetherian in codimension 1.
But $\mathfrak{t} \cap R = (\mathfrak{t} \cap S) \cap R = \mathfrak{q} \cap R$.
Therefore $\mathfrak{t} \in Z_{T/R}$ by Lemma \ref{lemma 1}.
\end{proof}

\begin{Lemma} \label{gameover}
Suppose $D_{S'} \subseteq Z_{S'/R}$, and set $T := T_{S'/R}$.
The morphism 
$$D_T \to D_S, \ \ \ \mathfrak{t} \mapsto \mathfrak{t} \cap S,$$ 
is well-defined and surjective.
\end{Lemma}

\begin{proof}
(i) We first claim that the map $D_T \to D_S$ is well-defined.
Let $\mathfrak{t} \in D_T$.
Then $\mathfrak{t} \in Z_{T/R}$ by Lemma \ref{t height 1}.
Thus there is a unique prime $\mathfrak{q}' \in \operatorname{Spec}S$ such that $\mathfrak{q}' \cap R = \operatorname{t} \cap R$, and $\mathfrak{q}' \in D_S$, by Lemma \ref{lemma 1}.
But then
$$(\mathfrak{t} \cap S) \cap R = \mathfrak{t} \cap R = \mathfrak{q}' \cap R.$$
Therefore, by the uniqueness of $\mathfrak{q}'$, we have
$$\mathfrak{t} \cap S = \mathfrak{q}' \in D_S.$$

(ii) We now claim that the map $D_T \to D_S$ is surjective.
Let $\mathfrak{q} \in D_S$.
Set $\mathfrak{p} := \mathfrak{q} \cap R$.
By Theorem \ref{depiction}, there is a prime $\mathfrak{t} \in \iota_{T/R}^{-1}(\mathfrak{p})$; we want to show that $\mathfrak{t} \in D_T$.
Indeed, assume to the contrary that $\mathfrak{t} \not \in D_T$, that is, $\operatorname{ht}_T(\mathfrak{t}) \geq 2$.
Then there is a prime $\mathfrak{t}' \in D_T$ properly contained in $\mathfrak{t}$.
Since $\operatorname{ht}_T(\mathfrak{t}') = 1$, we have $\mathfrak{t}' \in Z_{T/R}$ by Lemma \ref{t height 1}.
Thus the containment 
\begin{equation*} \label{p not p'}
\mathfrak{t}' \cap R \subset \mathfrak{t} \cap R = \mathfrak{p}
\end{equation*}
is proper, by Lemma \ref{lemma 1}.
Consequently, the containment
$$\mathfrak{t}' \cap S \subset \mathfrak{q}$$
is also proper.
But $\mathfrak{t}' \cap S$ is a nonzero prime of $S$ since $\mathfrak{t}'$ is a nonzero prime of $T$.
Therefore $\operatorname{ht}_S(\mathfrak{q}) \geq 2$, contrary to our choice of $\mathfrak{q}$.
\end{proof}

\begin{Lemma} \label{gameover2}
Suppose $D_{S'} \subseteq Z_{S'/R}$, and set $T := T_{S'/R}$.
Then
$$\bigcap_{\mathfrak{t} \in D_T} T_{\mathfrak{t}} = \bigcap_{\mathfrak{q} \in D_S}S_{\mathfrak{q}}.$$
\end{Lemma}

\begin{proof}
We have
\begin{equation} \label{out}
D_T \stackrel{\textsc{(i)}}{\subseteq} Z_{T/R} \stackrel{\textsc{(ii)}}{\subseteq} \tilde{U}_{T/R},
\end{equation}
where (\textsc{i}) holds by Lemma \ref{t height 1}, and (\textsc{ii}) holds by Lemma \ref{lemma 0}.3.

Let $\mathfrak{t} \in D_T$, and set $\mathfrak{q} := \mathfrak{t} \cap S \in \operatorname{Spec}S$.
Then, since $\mathfrak{t} \cap R = \mathfrak{q} \cap R$, we have
$$T_{\mathfrak{t}} \stackrel{\textsc{(i)}}{=} R_{\mathfrak{t} \cap R} = R_{\mathfrak{q} \cap R} \subseteq S_{\mathfrak{q}} \subseteq T_{\mathfrak{t}},$$
where (\textsc{i}) holds by (\ref{out}) and Lemma \ref{lemma 0}.3. 
Whence, $T_{\mathfrak{t}} = S_{\mathfrak{t} \cap S}$.
Therefore
$$\bigcap_{\mathfrak{t} \in D_T} T_{\mathfrak{t}} \stackrel{\textsc{(i)}}{=} \bigcap_{\mathfrak{t}\cap S \in D_S} S_{\mathfrak{t} \cap S} \stackrel{\textsc{(ii)}}{=} \bigcap_{\mathfrak{q} \in D_S} S_{\mathfrak{q}},$$
where (\textsc{i}) holds since $D_T \to D_S$ is well-defined by Lemma \ref{gameover}; and (\textsc{ii}) holds since $D_T \to D_S$ is surjective, again by Lemma \ref{gameover}.
\end{proof}

Suppose $R$ has a unique maximal depiction $T$, but is not noetherian in codimension 1. 
Then in general $R$ may admit a depiction $S$ for which the morphism
$$\iota_{T/S}: \operatorname{Spec}T \to \operatorname{Spec}S, \ \ \ \ \mathfrak{t} \mapsto \mathfrak{t} \cap S,$$
is not surjective; see Example \ref{not 1} below.
However, if $R$ is noetherian in codimension 1, then we have the following.

\begin{Proposition} \label{surjective morphism} 
Suppose $D_{S'} \subseteq Z_{S'/R}$, and set $T := T_{S'/R}$.
Then for any depiction $S$ of $R$, the morphism $\iota_{T/S}: \operatorname{Spec}T \to \operatorname{Spec}S$ is surjective.
\end{Proposition}

\begin{proof}
(i) We first claim that for each $\mathfrak{n} \in \operatorname{Max}S$, $\mathfrak{n}T \not = T$.

Assume to the contrary that there is some $\mathfrak{n} \in \operatorname{Max}S$ for which $\mathfrak{n}T = T$.
Let $\overbar{S_{\mathfrak{n}}}$ be the normalization of $S_{\mathfrak{n}}$.
Then
\begin{equation} \label{long}
\bigcap_{\mathfrak{t} \in D_T} T_{\mathfrak{t}}
\stackrel{\textsc{(i)}}{=} 
\bigcap_{\mathfrak{q} \in D_S} S_{\mathfrak{q}}
\subseteq
\bigcap_{\mathfrak{q} \in D_S: \mathfrak{q} \subseteq \mathfrak{n}} S_{\mathfrak{q}} 
\stackrel{\textsc{(ii)}}{=}
\bigcap_{\mathfrak{q} \in D_S: \mathfrak{q} \subseteq \mathfrak{n}} \left( S_{\mathfrak{n}} \right)_{\mathfrak{q}S_{\mathfrak{n}}}
\stackrel{\textsc{(iii)}}{\subseteq}
\bigcap_{\mathfrak{s} \in D_{S_{\mathfrak{n}}}} \left( S_{\mathfrak{n}} \right)_{\mathfrak{s}}
\stackrel{\textsc{(iv)}}{=}
\overbar{S_{\mathfrak{n}}}.
\end{equation}
Indeed, (\textsc{i}) holds by Lemma \ref{gameover2}.
(\textsc{ii}) holds since if $\mathfrak{q} \in \operatorname{Spec}S$ is contained in $\mathfrak{n}$, then 
\begin{equation*} \label{Sq =}
S_{\mathfrak{q}} = \left( S_{\mathfrak{n}}\right)_{\mathfrak{q}S_{\mathfrak{n}}}.
\end{equation*}
(\textsc{iii}) holds since if $\mathfrak{s} \in \operatorname{Spec}S_{\mathfrak{n}}$ has height $1$, then $\mathfrak{s} \cap S \in \operatorname{Spec}S$ also has height $1$, and $\mathfrak{s} \cap S \subseteq \mathfrak{n}$.
Finally, (\textsc{iv}) holds since $S_{\mathfrak{n}}$ is a noetherian domain \cite[Theorem 11.5.ii]{M}. 

By integrality, there is a prime ideal $\mathfrak{n}'$ of the normalization $\overbar{S_{\mathfrak{n}}}$ lying over $\mathfrak{n}$ \cite[Theorem 44]{K}.
Thus
\begin{equation*}
1 \in T = \mathfrak{n}T \subseteq \mathfrak{n}\left( \bigcap_{\mathfrak{t} \in D_T} T_{\mathfrak{t}} \right) 
\stackrel{\textsc{(i)}}{\subseteq}
\mathfrak{n} \overbar{S_{\mathfrak{n}}} 
\subseteq \mathfrak{n}',
\end{equation*}
where (\textsc{i}) holds by (\ref{long}). 
But then $1$ is in $\mathfrak{n}'$, a contradiction. 

(ii) We claim that the morphism of maximal spectra 
$$\operatorname{Max}T \to \operatorname{Max}S, \ \ \ \mathfrak{t} \mapsto \mathfrak{t} \cap S,$$
is surjective.
Let $\mathfrak{n} \in \operatorname{Spec}S$.
Then there is a maximal ideal $\mathfrak{t} \in \operatorname{Max}T$ containing $\mathfrak{n}T$ since $\mathfrak{n}T \not = T$ by Claim (i).
Whence
$$\mathfrak{n} \subseteq \mathfrak{n}T \cap S \subseteq \mathfrak{t} \cap S \not = S.$$
Therefore $\mathfrak{t} \cap S = \mathfrak{n}$ since $\mathfrak{n}$ is a maximal ideal.

(iii) $T$ is a finitely generated $k$-algebra by Proposition \ref{fg}, and $S$ is a finitely generated $k$-algebra since $S$ is a depiction.
Therefore $\iota_{T/S}$ is also surjective, by \cite[Lemma 3.6]{B3}.\footnote{The assumption in \cite[Lemma 3.6]{B3} that $k$ is uncountable is not necessary here since $S$ is a finitely generated $k$-algebra, rather than a countably generated $k$-algebra.}
\end{proof}

Note that, by the definition of geometric height, each $\mathfrak{q} \in \operatorname{Spec}S$ satisfies
$$\operatorname{ght}_R(\mathfrak{q} \cap R) \leq \operatorname{ht}_S(\mathfrak{q}).$$

\begin{Lemma} \label{goodgoodbye}
The following are equivalent:
\begin{enumerate} 
\item For each $\mathfrak{p} \in \operatorname{Spec}R$ there is some $\mathfrak{q} \in \iota_{S/R}^{-1}(\mathfrak{p})$ such that 
\begin{equation*}
\operatorname{ght}_R(\mathfrak{p}) = \operatorname{ht}_S(\mathfrak{q}).
\end{equation*}
\item For each $\mathfrak{q} \in \operatorname{Spec}S$ of minimal height in $\iota_{S/R}^{-1}(\mathfrak{q} \cap R)$, we have
$$\operatorname{ght}_R(\mathfrak{q} \cap R) = \operatorname{ht}_S(\mathfrak{q}).$$
\end{enumerate}
\end{Lemma}

\begin{proof}
(1) $\Rightarrow$ (2): Suppose $\mathfrak{q} \in \operatorname{Spec}S$ has minimal height in $\iota_{S/R}^{-1}(\mathfrak{p})$, where $\mathfrak{p} := \mathfrak{q} \cap R$.
By assumption (1), there is some $\mathfrak{q}' \in \iota^{-1}_{S/R}(\mathfrak{p})$ such that $\operatorname{ht}_S(\mathfrak{q}') = \operatorname{ght}_R(\mathfrak{p})$.
Therefore
$$\operatorname{ght}_R(\mathfrak{p}) \leq \operatorname{ht}_S(\mathfrak{q}) \leq \operatorname{ht}_S(\mathfrak{q}') = \operatorname{ght}_R(\mathfrak{p}).$$

(2) $\Rightarrow$ (1): Let $\mathfrak{p} \in \operatorname{Spec}R$.
Since $S$ is a depiction of $R$, we have $\iota_{S/R}^{-1}(\mathfrak{p}) \not = \emptyset$; choose $\mathfrak{q} \in \iota_{S/R}^{-1}(\mathfrak{p})$ of minimal height.
Then $\operatorname{ght}_R(\mathfrak{p}) = \operatorname{ht}_S(\mathfrak{q})$ by assumption (2).
\end{proof}

\begin{Definition} \rm{
If either (hence both) of the conditions in Lemma \ref{goodgoodbye} are satisfied, then we call $S$ a \textit{saturated} depiction of $R$.
}\end{Definition}

\begin{Lemma} 
If $S$ is a saturated depiction of $R$, then each $\mathfrak{q} \in \operatorname{Spec}S$ of minimal height in $\iota_{S/R}^{-1}(\mathfrak{q} \cap R)$ satisfies
$$\operatorname{gdim}(\mathfrak{q} \cap R) = \operatorname{dim}S/\mathfrak{q}.$$
\end{Lemma}

\begin{proof}
We have
$$\operatorname{gdim}(\mathfrak{q}\cap R) := \operatorname{dim}R - \operatorname{ght}_R(\mathfrak{q} \cap R) \stackrel{\textsc{(i)}}{=} \operatorname{dim}S - \operatorname{ht}_S(\mathfrak{q}) \stackrel{\textsc{(ii)}}{=} \operatorname{dim} S/\mathfrak{q},$$
where (\textsc{i}) holds by Lemma \ref{lemma 0}.1, and (\textsc{ii}) holds since $S$ is a finite type integral domain \cite[Proposition III.15]{S}.
\end{proof}

\begin{Lemma} \label{after}
Let $S$ be a noetherian integral domain, and let $U$ be a nonempty open subset of $\operatorname{Max}S$.
For each nonzero $\mathfrak{q} \in \operatorname{Spec}S$ there is some $\mathfrak{p} \in \operatorname{Spec}S$ contained in $\mathfrak{q}$ such that 
$$\mathcal{Z}(\mathfrak{p}) \cap U \not = \emptyset \ \ \ \text{ and } \ \ \ \operatorname{ht}(\mathfrak{p}) = \operatorname{ht}(\mathfrak{q}) -1.$$ 
\end{Lemma}

\begin{proof}
Fix $\mathfrak{q} \in \operatorname{Spec}S$.
Denote by $Q$ the set of primes $\mathfrak{p} \in \operatorname{Spec}S$ that are properly contained in $\mathfrak{q}$ and satisfy $\operatorname{ht}(\mathfrak{p}) = \operatorname{ht}(\mathfrak{q}) -1$.

Assume to the contrary that $\mathcal{Z}(\mathfrak{p}) \cap U = \emptyset$ for each $\mathfrak{p} \in Q$.
Then
$$\left( \cup_{\mathfrak{p} \in Q} \mathcal{Z}(\mathfrak{p}) \right) \cap U = \emptyset.$$ 
Whence $\overline{\cup_{\mathfrak{p} \in Q} \mathcal{Z}(\mathfrak{p})} \cap U = \emptyset$ since $U$ is open by Lemma \ref{lemma 0}.2.
Thus
$$\emptyset \not = U \subseteq \left(\overline{\cup_{\mathfrak{p} \in Q} \mathcal{Z}(\mathfrak{p})} \right)^c = \mathcal{Z}(\cap_{\mathfrak{p} \in Q} \mathfrak{p} )^c.$$
Therefore the ideal $I := \cap_{\mathfrak{p} \in Q} \mathfrak{p}$ is nonzero; say $0 \not = a \in I$.
In particular, $a \in \mathfrak{q}$ since for each $\mathfrak{p} \in Q$, $I \subseteq \mathfrak{p} \subset \mathfrak{q}$.

By \cite[Lemma 3.2]{F}, if $C$ is a noetherian integral domain, $\mathfrak{t} \in \operatorname{Spec}C$, and $0 \not = c \in \mathfrak{t}$, then there is a prime $\mathfrak{s} \in \operatorname{Spec}C$ such that $\operatorname{ht}(\mathfrak{s}) = \operatorname{ht}(\mathfrak{t}) - 1$ and $\mathfrak{s} \not \ni c$.
In our case we may take $C = S_{\mathfrak{q}}$, $\mathfrak{t} = \mathfrak{q}S_{\mathfrak{q}}$, and $c = a$.
Then there is a prime $\overbar{\mathfrak{p}} \in \operatorname{Spec}S_{\mathfrak{q}}$ such that
$$\operatorname{ht}_{S_{\mathfrak{q}}}(\overbar{\mathfrak{p}}) = \operatorname{ht}_{S_{\mathfrak{q}}}(\mathfrak{q}S_{\mathfrak{q}}) -1 = \operatorname{ht}_S(\mathfrak{q}) -1 \ \ \ \text{ and } \ \ \ \overbar{\mathfrak{p}} \not \ni a.$$
Set $\mathfrak{p} := \overbar{\mathfrak{p}} \cap S$.
Then $\operatorname{ht}_S(\mathfrak{p}) = \operatorname{ht}_{S_{\mathfrak{q}}}(\overbar{\mathfrak{p}})$ and $\mathfrak{p} \subset \mathfrak{q}$.
Thus $\mathfrak{p} \in Q$.
But $a \not \in \mathfrak{p}$ since $a \not \in \overbar{\mathfrak{p}}$. 
Therefore $a \not \in I$, contrary to our choice of $a$.
\end{proof}

\begin{Theorem} \label{geometric height}
Suppose $D_{S} \subseteq Z_{S/R}$.
Then $T_{S/R}$ is saturated. 
\end{Theorem}

\begin{proof}
Set $T:= T_{S/R}$.
Let $\mathfrak{p} \in \operatorname{Spec}R$.
Then there is a depiction $S'$ of $R$ such that for some $\mathfrak{q}$ in $\iota_{S'/R}^{-1}(\mathfrak{p})$, we have
$$\operatorname{ght}_R(\mathfrak{p}) = \operatorname{ht}_{S'}(\mathfrak{q}).$$

By Proposition \ref{surjective morphism}, $\iota_{T/S'}^{-1}(\mathfrak{q}) \not = \emptyset$; say $\mathfrak{t} \in \iota_{T/S'}^{-1}(\mathfrak{q})$.
Furthermore, $U_{T/R}$ is an open dense subset of $\operatorname{Max}T$, by Lemma \ref{lemma 0}.2.
Thus there is a prime $\mathfrak{t}' \in \operatorname{Spec}T$ properly contained in $\mathfrak{t}$, and maximal with respect to this inclusion, such that $\mathcal{Z}_T(\mathfrak{t}') \cap U_{T/R} \not = \emptyset$, by Lemma \ref{after}.

Set $m := \operatorname{ht}_T(\mathfrak{t})$. 
Consider a maximal chain of prime ideals of $T$ contained in $\mathfrak{t}'$,
\begin{equation*} \label{tm-1}
0 \subset \mathfrak{t}_1 \subset \mathfrak{t}_2 \subset \cdots \subset \mathfrak{t}_{m-1} = \mathfrak{t}' \subset \mathfrak{t}_m = \mathfrak{t},
\end{equation*} 
and the corresponding chain of prime ideals of $S'$,
\begin{equation} \label{chain 2}
0 \subset \mathfrak{t}_1 \cap S' \subseteq \mathfrak{t}_2 \cap S' \subseteq \cdots \subseteq \mathfrak{t}_m \cap S' = \mathfrak{q}.
\end{equation}
We claim that the chain (\ref{chain 2}) is strict.
Indeed, assume to the contrary that there is some $1 \leq i < m$ for which $\mathfrak{t}_i \cap S' = \mathfrak{t}_{i+1} \cap S'$.
Then $\mathfrak{t}_i \cap R = \mathfrak{t}_{i+1} \cap R$.
Furthermore, $\mathcal{Z}_T(\mathfrak{t}_i) \cap U_{T/R} \not = \emptyset$ since $\mathcal{Z}_T(\mathfrak{t}_{m-1}) \cap U_{T/R} \not = \emptyset$.
But then $\mathfrak{t}_i = \mathfrak{t}_{i+1}$ by Lemmas \ref{lemma 0}.3 and \ref{lemma 0}.4, a contradiction.

It follows that
$$\operatorname{ght}_R(\mathfrak{p}) \stackrel{\textsc{(i)}}{\leq} \operatorname{ht}_T(\mathfrak{t}) = m \stackrel{\textsc{(ii)}}{\leq} \operatorname{ht}_{S'}(\mathfrak{q}) = \operatorname{ght}_R(\mathfrak{p}),$$
where (\textsc{i}) holds since  $\mathfrak{t} \cap R = \mathfrak{t} \cap S' \cap R = \mathfrak{q} \cap R = \mathfrak{p}$, and (\textsc{ii}) holds since the chain (\ref{chain 2}) is strict.
Therefore $\operatorname{ght}_R(\mathfrak{p}) = \operatorname{ht}_T(\mathfrak{t})$.
\end{proof}

\begin{Proposition} \label{independent}
Each codimension 1 subvariety of $\operatorname{Max}S$ intersects $U_{S/R}$ if and only if each codimension 1 subvariety of $\operatorname{Max}S'$ intersects $U_{S'/R}$:
$$D_S \subseteq Z_{S/R} \ \ \ \Longleftrightarrow \ \ \ D_{S'} \subseteq Z_{S'/R}.$$
In particular, the definition of `noetherian in codimension 1' is independent of the choice of depiction.
\end{Proposition}

\begin{proof}
Suppose $D_S \subseteq Z_{S/R}$, and consider $\mathfrak{q} \in D_{S'}$; we want to show that $\mathfrak{q} \in Z_{S'/R}$.

Set $T := T_{S/R}$ and $\mathfrak{p} := \mathfrak{q} \cap R$.
By Theorems \ref{depiction} and \ref{geometric height}, $T$ is a saturated depiction of $R$.
Thus there is some $\mathfrak{t} \in \iota_{T/R}^{-1}(\mathfrak{p})$ such that $\operatorname{ht}_T(\mathfrak{t}) = \operatorname{ght}_R(\mathfrak{p})$.
Therefore
$$1 \stackrel{\textsc{(i)}}{\leq} \operatorname{ht}_T(\mathfrak{t}) = \operatorname{ght}_R(\mathfrak{p}) \stackrel{\textsc{(ii)}}{\leq} \operatorname{ht}_{S'}(\mathfrak{q}) = 1,$$
where (\textsc{i}) holds since $\mathfrak{t} \not = 0$, and (\textsc{ii}) holds since $\mathfrak{q} \cap R = \mathfrak{p}$.
Whence $\operatorname{ht}_T(\mathfrak{t}) = 1$.
Thus $\mathfrak{t} \in Z_{T/R}$ by Lemma \ref{t height 1}.
Therefore there is a unique prime $\mathfrak{q}' \in \operatorname{Spec}S'$ such that $\mathfrak{q}' \cap R = \mathfrak{t} \cap R$, and $\mathfrak{q}' \in Z_{S'/R}$, by Lemma \ref{lemma 1}.
But
$$\mathfrak{q} \cap R = \mathfrak{p} = \mathfrak{t} \cap R = \mathfrak{q}' \cap R.$$
It follows that $\mathfrak{q}' = \mathfrak{q}$, by the uniqueness of $\mathfrak{q}'$.
Therefore $\mathfrak{q} \in Z_{S'/R}$.
\end{proof}

\begin{Theorem} \label{normal theorem}
Suppose $R$ is noetherian in codimension $1$.
Let $S$ and $S'$ be arbitrary depictions of $R$.
Then
\begin{equation*} \label{bye}
T := T_{S/R} = T_{S'/R},
\end{equation*}
and $T$ is the unique maximal depiction of $R$. 
Furthermore, $T$ is contained in the normalization of each depiction of $R$,
\begin{equation*} \label{byebye}
S \subseteq T \subseteq \overbar{S}.
\end{equation*}
In particular, if $S = \overbar{S}$ is normal, then $S$ is the unique maximal depiction of $R$, as well as the unique normal depiction of $R$.
\end{Theorem}

\begin{proof}
The overrings $T_{S/R}$ and $T_{S'/R}$ are both depictions of $R$ by Proposition \ref{independent} and Theorem \ref{depiction}.
But $T_{S/R}$ and $T_{S'/R}$ each contain every depiction of $R$, by Proposition \ref{contains}.
Therefore $T_{S/R} = T_{S'/R}$.
Finally, the inclusion $T \subseteq \overbar{S}$ holds by Lemma \ref{contained in normalization}.
\end{proof}

\section{Saturated depictions of coordinate rings with a unique positive dimensional closed point} \label{specialsection}

Rings of the form $R = k + I$, where $I$ is an ideal of a finite type integral domain $S$, form a particularly nice class of nonnoetherian rings in the study of nonnoetherian geometry.
It was shown in \cite[Corollary 1.3]{B2} that if $I$ is a proper nonzero non-maximal radical ideal of $S$, then the following are equivalent:
\begin{enumerate}
 \item $\operatorname{dim}S/I \geq 1$.
 \item $R$ is nonnoetherian.
 \item $R$ is depicted by $S$.
\end{enumerate}
Furthermore, if $R$ is nonnoetherian, then 
\begin{equation*} \label{gbye}
U_{S/R} = \mathcal{Z}_S(I)^c.
\end{equation*}

In the following, we do not assume that $R$ is noetherian in codimension $1$.

\begin{Theorem} \label{wind}
Let $I$ be a nonzero radical ideal of $S$ such that $\operatorname{dim}S/I \geq 1$, and set $R := k + I$.
If $S$ is a unique factorization domain or $\operatorname{ht}_S(I) = 1$, then $S$ is a saturated depiction of $R$.
\end{Theorem}

\begin{proof}
By \cite[Corollary 1.3]{B2}, $S$ is a depiction of $R$ since $I$ is a nonzero radical ideal of $S$ satisfying $\operatorname{dim}S/I \geq 1$. 

(i) First suppose $S$ is a UFD and $\operatorname{ht}_S(I) \geq 2$. 
Then $D_S \subseteq Z_{S/R}$, since $U_{S/R} = \mathcal{Z}_S(I)^c$.
Furthermore,
$$T_{S/R} := \Gamma(U_{S/R}) = \Gamma(\mathcal{Z}_S(I)^c) \stackrel{\textsc{(i)}}{=} S,$$ 
where (\textsc{i}) holds since $S$ is a UFD and $\operatorname{ht}_S(I) \geq 2$.
Therefore $S$ is saturated by Theorem \ref{geometric height}.

(ii) Now suppose $\operatorname{ht}_S(I) = 1$.
Consider $\mathfrak{q} \in \iota_{S/R}^{-1}(\mathfrak{q}\cap R)$ with minimal height.
Then either $\mathfrak{q} \in Z_{S/R}$, or $\mathfrak{q}$ has minimal height such that $\mathcal{Z}_S(\mathfrak{q}) \subseteq \mathcal{Z}_S(I)$.

If $\mathfrak{q} \in Z_{S/R}$, then by Proposition \ref{lemma 1.1}, 
$$\operatorname{ght}_R(\mathfrak{q}\cap R) = \operatorname{ht}_S(\mathfrak{q}).$$
So suppose $\mathfrak{q}$ has minimal height such that $\mathcal{Z}_S(\mathfrak{q}) \subseteq \mathcal{Z}_S(I)$.
Then $\mathfrak{q} \supseteq I$ since $\mathfrak{q}$ and $I$ are radical ideals of $S$.
Hence, $\mathfrak{q} \cap R = I$ since $I$ is a maximal ideal of $R$.
Thus
$$1 \stackrel{\textsc{(i)}}{\leq} \operatorname{ght}_R(I) \stackrel{\textsc{(ii)}}{\leq} \operatorname{ht}_S(\mathfrak{q}) \stackrel{\textsc{(iii)}}{=} \operatorname{ht}_S(I) \stackrel{\textsc{(iv)}}{=} 1,$$
where (\textsc{i}) holds since $I \not = 0$ and $S$ is an integral domain; (\textsc{ii}) holds since $\mathfrak{q} \cap R = I$; (\textsc{iii}) holds since $\mathfrak{q}$ is a minimal prime over $I$ of minimal height; and (\textsc{iv}) holds by assumption.
Consequently,
$$\operatorname{ght}_R(I) = \operatorname{ht}_S(\mathfrak{q}).$$
Therefore $S$ is saturated.
\end{proof}

\section{Examples} \label{examplesection}

\begin{Example} \label{j geq 1} \rm{
Consider the family of algebras
$$S_j := k[x,y,xz,yz,xz^2, yz^2, \ldots, xz^{j-1},yz^{j-1}, z^j] \ \ \ \text{ and } \ \ \ R := k + (x,y)S_1,$$
where $j \geq 1$, and $(x,y)S_1$ is the ideal of $S_1 = k[x,y,z]$ generated by $x$ and $y$.
By Theorem \ref{wind}, $S_1$ is a saturated depiction of $R$ with 
$$U_{S_1/R} = \mathcal{Z}_{S_1}(x,y)^c.$$ 
Since each $2$-dimensional subvariety of $\operatorname{Max}S_1 = \mathbb{A}_k^3$ intersects the complement of the line $\mathcal{Z}_{S_1}(x,y)$, $R$ is noetherian in codimension $1$.
Furthermore, since $S_1$ is a polynomial ring, it is normal.
Therefore $S_1$ is the unique maximal depiction of $R$, as well as the unique normal depiction of $R$, by Theorem \ref{normal theorem}.

We will show that each $S_j \subseteq S_1$ is also a depiction of $R$.
Fix $j \geq 1$.  

We first claim that $\iota_{S_j/R}: \operatorname{Spec}S_j \to \operatorname{Spec}R$ is surjective.
Let $\mathfrak{p} \in \operatorname{Spec}R$.
Since $S_1$ is a depiction of $R$, there is a prime $\mathfrak{q} \in \operatorname{Spec}S_1$ such that $\mathfrak{q} \cap R = \mathfrak{p}$.
But $R \subset S_j \subseteq S_1$.
Therefore the prime $\mathfrak{q}' := \mathfrak{q} \cap S_j \in \operatorname{Spec}S_j$ satisfies $\mathfrak{q}' \cap R = \mathfrak{p}$, proving our claim. 
 
We now claim that (\ref{condition}) in Definition \ref{depiction def} holds.
Let $\mathfrak{n} \in \operatorname{Max}S_j$ be such that $R_{\mathfrak{n}\cap R}$ is noetherian.
$S_1$ is a finitely generated $S_j$-module with generating set $\{1, z,z^2, \ldots, z^{j-1} \}$.
Thus, by Nakayama's lemma, $\mathfrak{n}S_1 \not = S_1$.
Therefore there is some $\mathfrak{t} \in \operatorname{Max}S_1$ such that $\mathfrak{t} \cap S_j = \mathfrak{n}$. 
Furthermore, since $S_1$ is a depiction of $R$ and $R_{\mathfrak{t} \cap R} = R_{\mathfrak{n} \cap R}$ is noetherian, we have $(S_1)_{\mathfrak{t}} = R_{\mathfrak{t} \cap R}$.
Thus,
$$(S_j)_{\mathfrak{n}} = (S_j)_{\mathfrak{t} \cap S_j} \subseteq (S_1)_{\mathfrak{t}} = R_{\mathfrak{t} \cap R} = R_{\mathfrak{n} \cap R} \subseteq (S_j)_{\mathfrak{n}}.$$
Whence $(S_j)_{\mathfrak{n}} = R_{\mathfrak{n} \cap R}$, proving our claim. 
$S_j$ is therefore a depiction of $R$.
}\end{Example}

In the following example, we show that if $R$ is not noetherian in codimension $1$, then $R$ may admit a unique maximal depiction $T$ which is not saturated. 
This example demonstrates that unique maximal depictions, when they exist, are not necessarily the `right' depictions for some nonnoetherian rings.

\begin{Example} \label{not 1} \rm{ 
Consider the algebras
$$T := k[x,x^{-1}, y], \ \ \ \ S := k[x,y], \ \ \ \ R := k + I,$$
where $I := x(x-1,y)S$.
We will show that $T$ is the unique maximal depiction of $R$, but is not saturated.

Since $\operatorname{dim}S/I = 1$, $S$ is a saturated depiction of $R$, by Theorem \ref{wind}.
However, since the localization $R_{xS \cap R} = R_I$ is not noetherian, $R$ is not noetherian in codimension $1$, by Proposition \ref{independent}.
Clearly $T$ is also a depiction of $R$, with
\begin{equation} \label{check}
U_{T/R} = \mathcal{Z}_T(\mathfrak{t}_0)^c,
\end{equation}
where $\mathfrak{t}_0 := (x-1,y)T \in \operatorname{Max}T$.

$\operatorname{Max}R$ may therefore be viewed either as the plane $\operatorname{Max}S = \mathbb{A}_k^2$ where the point $\mathcal{Z}_S(x-1,y)$ and line $\mathcal{Z}_S(x)$ are together identified as a single closed point; or as the open subset of the plane 
$$\operatorname{Max}T \cong \mathcal{Z}_S(x)^c.$$
From the perspective of $T$, all the closed points of $\operatorname{Max}R$, including $I$ itself, appear zero-dimensional.

We claim that $T$ is the unique maximal depiction of $R$.
Indeed, let $S'$ be any depiction of $R$.
Then
\begin{multline*}
S' \subseteq \bigcap _{\mathfrak{q} \in Z_{S'/R}} S'_{\mathfrak{q}} 
\stackrel{\textsc{(i)}}{=} \bigcap_{\mathfrak{q} \in Z_{S'/R}} R_{\mathfrak{q} \cap R}
\stackrel{\textsc{(ii)}}{=} \bigcap_{\mathfrak{t} \in Z_{T/R}} R_{\mathfrak{t} \cap R}
\stackrel{\textsc{(iii)}}{=} \bigcap_{\mathfrak{t} \in Z_{T/R}} T_{\mathfrak{t}} \\
\stackrel{\textsc{(iv)}}{=} \bigcap_{\mathfrak{t} \in \operatorname{Spec}T \setminus \{ \mathfrak{t}_0 \}} T_{\mathfrak{t}} 
= \bigcap_{\mathfrak{t} \in \operatorname{Spec}T} T_{\mathfrak{t}} = T,
\end{multline*}
where (\textsc{i}) and (\textsc{iii}) hold by Lemma \ref{lemma 0}.3; (\textsc{ii}) holds by Lemma \ref{lemma 1}; and (\textsc{iv}) holds by (\ref{check}).
Therefore $S'$ is contained in $T$.

We now claim that $T$ is not saturated.
Set $\mathfrak{q} = xS$; then
$$\mathfrak{q} \cap R = \mathfrak{t}_0 \cap R = I.$$ 
Furthermore,
$$1 \stackrel{\textsc{(i)}}{\leq} \operatorname{ght}_R(I) \leq \operatorname{ht}_{S}(\mathfrak{q}) = 1 < 2 = \operatorname{ht}_T(\mathfrak{t}_0),$$
where (\textsc{i}) holds since $I \not = 0$.
 
Let $\mathfrak{t} \in \operatorname{Spec}T \setminus \{ \mathfrak{t}_0 \}$.
Since $\mathfrak{t}_0$ is a maximal ideal of $T$, (\ref{check}) implies $\mathfrak{t} \in Z_{T/R}$.
Whence $\mathfrak{t} \in \tilde{U}_{T/R}$ by Lemma \ref{lemma 0}.3.
Thus $\mathfrak{t} \not \in \iota_{T/R}^{-1}(I)$ by Lemma \ref{lemma 0}.4.
Therefore
$$\iota_{T/R}^{-1}(I) = \left\{ \mathfrak{t}_0 \right\}.$$
It follows that there is a prime $\mathfrak{p} \in \operatorname{Spec}R$, namely $I$, such that for each $\mathfrak{t} \in \iota_{T/R}^{-1}(\mathfrak{p})$,
$$\operatorname{ght}_R(\mathfrak{p}) < \operatorname{ht}_T(\mathfrak{t}).$$
}\end{Example}

In our final example, we show that $R$ need not be noetherian in codimension $1$ in order for $R$ to admit a saturated unique maximal depiction.

\begin{Example} \rm{
Consider the algebras
$$S := k[x,y] \ \ \ \text{ and } \ \ \ R := k + xS.$$
By Theorem \ref{wind}, $S$ is a saturated depiction of $R$, with 
\begin{equation*} \label{last}
U_{S/R} = \mathcal{Z}_S(x)^c.
\end{equation*}
Furthermore, by Proposition \ref{independent}, $R$ is not noetherian in codimension $1$ since the localization $R_{xS \cap R}$ is not noetherian.

We claim that $S$ is the unique maximal depiction of $R$.
Let $S'$ be any depiction of $R$.
Then by Proposition \ref{contains}, 
$$S' \subseteq \Gamma(U_{S/R}) = \Gamma(\mathcal{Z}_S(x)^c) = S[x^{-1}].$$
However, $\iota_{S[x^{-1}]/R}$ is not surjective since the ideal $xS \in \operatorname{Max}R$ does not have a pre-image in $S[x^{-1}]$.
Furthermore, $\iota_{S'/R}^{-1}(xS) \not = \emptyset$ since $S'$ is a depiction of $R$.
It follows that $S' \subseteq S$.
$S$ is therefore a saturated unique maximal depiction of $R$, even though $R$ is not noetherian in codimension $1$.
}\end{Example}

\textbf{Acknowledgments.}  The author would like to thank an anonymous referee for useful comments.
This article was completed while the author was a research fellow at the Heilbronn Institute for Mathematical Research at the University of Bristol. 

\bibliographystyle{hep}
\def\cprime{$'$} \def\cprime{$'$}

\end{document}